\documentclass[oneside,english]{amsart}
\usepackage[T1]{fontenc}
\usepackage[latin9]{inputenc}
\usepackage{prettyref}
\usepackage{enumitem}
\usepackage{amstext}
\usepackage{amsthm}
\usepackage{amssymb}
\usepackage[all]{xy}

\makeatletter
\numberwithin{equation}{section}
\numberwithin{figure}{section}
\theoremstyle{plain}
\newtheorem{thm}{\protect\theoremname}[section]
\theoremstyle{plain}
\newtheorem{prop}[thm]{\protect\propositionname}
\theoremstyle{definition}
\newtheorem{example}[thm]{\protect\examplename}
\newlist{casenv}{enumerate}{4}
\setlist[casenv]{leftmargin=*,align=left,widest={iiii}}
\setlist[casenv,1]{label={{\itshape\ \casename} \arabic*.},ref=\arabic*}
\setlist[casenv,2]{label={{\itshape\ \casename} \roman*.},ref=\roman*}
\setlist[casenv,3]{label={{\itshape\ \casename\ \alph*.}},ref=\alph*}
\setlist[casenv,4]{label={{\itshape\ \casename} \arabic*.},ref=\arabic*}
\theoremstyle{plain}
\newtheorem{lem}[thm]{\protect\lemmaname}
\theoremstyle{plain}
\newtheorem{cor}[thm]{\protect\corollaryname}
\theoremstyle{remark}
\newtheorem{rem}[thm]{\protect\remarkname}
\theoremstyle{definition}
\newtheorem{problem}[thm]{\protect\problemname}

\usepackage{mathtools}
\usepackage{prettyref}
\usepackage{braket}
\usepackage{tikz}
\usetikzlibrary{patterns}

\newrefformat{prop}{Proposition \ref{#1}}
\newrefformat{cor}{Corollary \ref{#1}}
\newrefformat{exa}{Example \ref{#1}}
\newrefformat{fig}{Figure \ref{#1}}

\global\long\def\ns#1{\prescript{\ast}{}{#1}}
\newcommand{\FIN}{\operatorname{FIN}}
\newcommand{\INF}{\operatorname{INF}}

\@ifundefined{showcaptionsetup}{}{%
 \PassOptionsToPackage{caption=false}{subfig}}
\usepackage{subfig}
\makeatother

\usepackage{babel}
\providecommand{\casename}{Case}
\providecommand{\corollaryname}{Corollary}
\providecommand{\examplename}{Example}
\providecommand{\lemmaname}{Lemma}
\providecommand{\problemname}{Problem}
\providecommand{\propositionname}{Proposition}
\providecommand{\remarkname}{Remark}
\providecommand{\theoremname}{Theorem}

\begin{document}
\title{A \lowercase{nonstandard invariant of coarse spaces}}
\author{Takuma Imamura}
\address{Research Institute for Mathematical Sciences\\
Kyoto University\\
Kitashirakawa Oiwake-cho, Sakyo-ku, Kyoto 606-8502, Japan}
\email{timamura@kurims.kyoto-u.ac.jp}
\begin{abstract}
We construct a set-valued invariant $\iota\left(X,\xi\right)$ of
pointed coarse spaces $\left(X,\xi\right)$ by using nonstandard analysis.
The invariance under coarse equivalence is established. A sufficient
condition for the invariant to be of cardinality $\leq1$ is provided.
Miller \emph{et al.} \cite{MSM10} and subsequent researchers have
introduced a similar but standard set-valued coarse invariant $\sigma\left(X,\xi\right)$
of pointed metric spaces $\left(X,\xi\right)$. In order to compare
these two invariants, we construct a natural transformation $\omega_{\left(X,\xi\right)}$
from $\sigma\left(X,\xi\right)$ to $\iota\left(X,\xi\right)$. The
surjectivity of $\omega_{\left(X,\xi\right)}$ is proved for all proper
geodesic spaces $\left(X,\xi\right)$.
\end{abstract}

\maketitle

\section{Introduction}

Small-scale topology focuses on \emph{fine} structures of spaces and
maps, such as openness, continuity, and convergence. In contrast,
large-scale (coarse) topology does not care about fine structures,
and instead focuses on \emph{coarse} or \emph{asymptotic} structures
of spaces and maps, such as boundedness, bornologicity, and divergence.
Large-scale concepts appear in many fields of mathematics, e.g., functional
analysis \cite{Bou07,HN77}, geometric group theory \cite{Bow06,Roe03},
and infinite combinatorics \cite{PB03,PZ07}. Roe \cite{Roe03} developed
a general and unified framework for large-scale topology, called coarse
spaces. A \emph{coarse structure} on a set $X$ is a family $\mathcal{C}_{X}$
of binary relations on $X$ which satisfies the following conditions:
\begin{enumerate}
\item $\Delta_{X}:=\set{\left(x,x\right)|x\in X}\in\mathcal{C}_{X}$;
\item $E\subseteq F\in\mathcal{C}_{X}\implies E\in\mathcal{C}_{X}$;
\item $E,F\in\mathcal{C}_{X}\implies E\cup F\in\mathcal{C}_{X}$;
\item $E,F\in\mathcal{C}_{X}\implies E\circ F\in\mathcal{C}_{X}$;
\item $E\in\mathcal{C}_{X}\implies E^{-1}\in\mathcal{C}_{X}$.
\end{enumerate}
A set equipped with a coarse structure is called a \emph{coarse space}.
For example, given a pseudometric metric $d_{X}$ on a set $X$, the
family of all binary relations $E$ on $X$ with $\sup_{\left(x,y\right)\in E}d_{X}\left(x,y\right)<\infty$
is a coarse structure on $X$, called the \emph{bounded coarse structure}.

Miller \emph{et al.} \cite{MSM10} introduced a set-valued invariant
$\sigma\left(X,\xi\right)$ of pointed metric spaces $\left(X,\xi\right)$
with a certain property (called $\sigma$-stability). Fox \emph{et
al.} \cite{FLL11} proved that $\sigma\left(X,\xi\right)$ is invariant
under coarse equivalence. DeLyser \emph{et al.} \cite{DLW11} extended
it to general pointed metric spaces via the direct limit construction.
The definition of $\sigma\left(X,\xi\right)$ can be obviously extended
to general pointed coarse spaces. Roughly speaking, $\sigma\left(X,\xi\right)$
counts \emph{the ways to tend to infinity from $\xi$} up to a certain
equivalence. Analogously, it is natural to consider the number of\emph{
ideal points infinitely far away from $\xi$} up to some equivalence.

In this paper, we employ Robinson's NonStandard Analysis (NSA) to
realise the above idea. NSA is a powerful framework for finding and
proving theorems, refactoring known theories, constructing mathematical
objects, and crystallising intuitive ideas, based on the existence
of saturated models of set theory. The basic strategy of NSA is to
enrich the mathematical world by adding ideal entities such as infinitesimals
(infinitely small quantities). We refer to \cite{Dav05,LW15,Rob66}
for NSA and \cite{CK90,KR04} for the foundations of NSA.

This paper is organised as follows. In \prettyref{sec:Preliminaries},
we review the treatment of coarse spaces by means of nonstandard analysis.
In \prettyref{sec:Invariant-iota}, we construct a set-valued coarse
invariant $\iota\left(X,\xi\right)$ for each pointed coarse space
$\left(X,\xi\right)$. We consider an asymptotic property of pointed
coarse spaces, called ``non-scattering at infinity'', and show that
it is a sufficient condition for $\iota\left(X,\xi\right)$ to be
of cardinality $\leq1$. In \prettyref{sec:Comparison-with-sigma},
we discuss the relationship between $\sigma\left(X,\xi\right)$ and
$\iota\left(X,\xi\right)$ by considering a natural transformation
$\omega\colon\sigma\to\iota$. The surjectivity of $\omega_{\left(X,\xi\right)}\colon\sigma\left(X,\xi\right)\to\iota\left(X,\xi\right)$
is proved for all proper geodesic spaces $\left(X,\xi\right)$. Finally,
in \prettyref{sec:Some-open-problems}, we pose some open problems.
For instance, it is open whether the map $\omega_{\left(X,\xi\right)}$
can be non-injective.

\section{\label{sec:Preliminaries}Preliminaries}

In this section, we recall basic definitions and results in nonstandard
large-scale topology following \cite{Ima19}.

Let $\left(X,\mathcal{C}_{X}\right)$ be a coarse space. Consider
the nonstandard extension $\left(\ns{X},\ns{\mathcal{C}_{X}}\right)$.
We say that two points $x,y\in\ns{X}$ are \emph{finitely close} (and
denote by $x\sim_{X}y$) if $\left(x,y\right)\in\ns{E}$ for some
$E\in\mathcal{C}_{X}$. In other words, the \emph{finite closeness
relation} $\sim_{X}$ is a binary relation on $\ns{X}$ defined as
the union $\bigcup_{E\in\mathcal{C}_{X}}\ns{E}$. For $\xi\in X$
let $\FIN\left(X,\xi\right):=\set{x\in\ns{X}|x\sim_{X}\xi}$. For
example, if $\mathcal{C}_{X}$ is induced by a pseudometric $d_{X}$,
then $x\sim_{X}y$ if and only if $\ns{d}_{X}\left(x,y\right)$ is
finite (in the sense of nonstandard analysis).

We then obtain the following nonstandard characterisations.
\begin{prop}[{\cite[Proposition 3.4]{Ima19}}]
\label{prop:control}A binary relation $E$ on a coarse space $X$
is controlled (i.e. $\in\mathcal{C}_{X}$) if and only if $\ns{E}\subseteq{\sim_{X}}$.
\end{prop}

\begin{proof}
The ``only if'' part is trivial. Suppose $\ns{E}\subseteq{\sim_{X}}$.
By the saturation principle, there exists an $F\in\ns{\mathcal{C}_{X}}$
such that ${\sim_{X}}\subseteq F$ (see also \cite[Lemma 3.3]{Ima19}),
so $\ns{E}\subseteq F$. By the transfer principle, $\ns{E}\in\ns{\mathcal{C}_{X}}$.
Again by transfer, we have $E\in\mathcal{C}_{X}$.
\end{proof}
\begin{prop}[{\cite[Proposition 3.10]{Ima19}}]
\label{prop:bounded}A subset $B$ of a coarse space $X$ is bounded
if and only if $x\sim_{X}y$ for all $x,y\in\ns{B}$.
\end{prop}

\begin{proof}
Only if: $B$ is bounded, i.e. $B\times B\in\mathcal{C}_{X}$. Then
$\ns{B}\times\ns{B}=\ns{\left(B\times B\right)}\subseteq{\sim_{X}}$,
where the equality follows from transfer. Hence $x\sim_{X}y$ holds
for all $x,y\in\ns{B}$.

If: $x\sim_{X}y$ holds for all $x,y\in\ns{B}$, i.e. $\ns{B}\times\ns{B}\subseteq{\sim_{X}}$.
By \prettyref{prop:control}, $B\times B\in\mathcal{C}_{X}$, i.e.
$B$ is bounded.
\end{proof}
\begin{prop}[{\cite[Corollary 3.13]{Ima19}}]
\label{prop:coarse-connectedness}A coarse space $X$ is coarsely
connected if and only if $x\sim_{X}y$ for all $x,y\in X$.
\end{prop}

\begin{proof}
The coarse space $X$ is coarsely connected if and only if $\set{x,y}$
is bounded for all $x,y\in X$. Apply \prettyref{prop:bounded} to
the latter condition.
\end{proof}
\begin{prop}[{\cite[Theorem 3.23]{Ima19}}]
\label{prop:bornologous}A map $f\colon X\to Y$ between coarse spaces
is bornologous if and only if $\ns{f}\colon\ns{X}\to\ns{Y}$ preserves
finite closeness, i.e., $x\sim_{X}x'$ implies $\ns{f}\left(x\right)\sim_{Y}\ns{f}\left(x'\right)$.
\end{prop}

\begin{proof}
Only if: let $x,x'\in\ns{X}$ and suppose $x\sim_{X}x'$. Choose an
$E\in\mathcal{C}_{X}$ such that $\left(x,x'\right)\in\ns{E}$. Since
$f$ is bornologous, $F:=\left(f\times f\right)\left(E\right)\in\mathcal{C}_{Y}$.
By transfer, we have $\left(\ns{f}\left(x\right),\ns{f}\left(x'\right)\right)\in\ns{F}$.
Hence $\ns{f}\left(x\right)\sim_{Y}\ns{f}\left(x'\right)$.

If: let $E\in\mathcal{C}_{X}$. Then $\ns{\left(\left(f\times f\right)\left(E\right)\right)}=\left(\ns{f}\times\ns{f}\right)\left(\ns{E}\right)\subseteq\left(\ns{f}\times\ns{f}\right)\left(\sim_{X}\right)\subseteq{\sim_{Y}}$.
By \prettyref{prop:control}, $\left(f\times f\right)\left(E\right)\in\mathcal{C}_{Y}$.
Hence $f$ is bornologous, because $E$ was arbitrary.
\end{proof}
\begin{prop}[{\cite[Theorem 2.26]{Ima19}}]
\label{prop:proper}A map $f\colon\left(X,\xi\right)\to\left(Y,\eta\right)$
between pointed coarse spaces is proper at the base point (i.e. the
inverse image of each bounded set of $Y$ containing $\eta$ is bounded
in $X$) if and only if $\ns{f}^{-1}\left(\FIN\left(Y,\eta\right)\right)\subseteq\FIN\left(X,\xi\right)$.
\end{prop}

\begin{proof}
Only if: $\ns{f}^{-1}\left(\FIN\left(Y,\eta\right)\right)=\bigcup_{\eta\in B\subseteq_{\mathrm{bounded}}Y}\ns{f}^{-1}\left(\ns{B}\right)\subseteq\bigcup_{\xi\in A\subseteq_{\mathrm{bounded}}X}\ns{A}=\FIN\left(X,\xi\right)$.

If: let $B$ be a bounded set of $Y$ that contains $\eta$. Then
$\ns{B}\subseteq\FIN\left(Y,\eta\right)$ by \prettyref{prop:bounded}.
By assumption, $\ns{\left(f^{-1}\left(B\right)\right)}=\ns{f}^{-1}\left(\ns{B}\right)\subseteq\ns{f}^{-1}\left(\FIN\left(Y,\eta\right)\right)\subseteq\FIN\left(X,\xi\right)$.
Hence $f^{-1}\left(B\right)$ is bounded by \prettyref{prop:bounded}.
\end{proof}
\begin{prop}
\label{prop:bornotopy}Two maps $f,g\colon X\to Y$ between coarse
spaces are bornotopic (or close) if and only if $\ns{f}\left(x\right)\sim_{Y}\ns{g}\left(x\right)$
for all $x\in\ns{X}$.
\end{prop}

\begin{proof}
Recall that $f$ and $g$ are bornotopic if and only if $\set{\left(f\left(x\right),g\left(x\right)\right)|x\in X}\in\mathcal{C}_{Y}$.
Then apply \prettyref{prop:control}.
\end{proof}
We say that a coarse space $X$ is \emph{coarsely Archimedean} if
$X\times X=\bigcup_{n\in\mathbb{N}}E^{n}$ holds for some $E\in\mathcal{C}_{X}$,
where $E^{n}$ refers to the $n$-fold composition of $E$. Note that
this notion is the large-scale counterpart of the Archimedean property
of uniform spaces (Hu \cite{Hu47}). A (possibly external) subset
$B$ of $\ns{X}$ is said to be \emph{macrochain-connected} if for
any $x,y\in B$, there exists an internal hyperfinite sequence $\set{s_{i}}_{i\leq n}$
in $B$, where $n\in\ns{\mathbb{N}}$, such that $s_{0}=x$, $s_{n}=y$
and $s_{i}\sim_{X}s_{i+1}$ for all $i<n$. For standard sets, these
two notions are equivalent.
\begin{prop}
$X$ is coarsely Archimedean if and only if $\prescript{\ast}{}{X}$
is macrochain-connected.
\end{prop}

\begin{proof}
Only if: choose an $E\in\mathcal{C}_{X}$ such that $X\times X=\bigcup_{n\in\mathbb{N}}E^{n}$.
For any $x,y\in\ns{X}$, by transfer, there exists an $n\in\prescript{\ast}{}{\mathbb{N}}$
such that $\left(x,y\right)\in\ns{E^{n}}\subseteq{\sim_{X}}^{n}$.
Hence $\prescript{\ast}{}{X}$ is macrochain-connected.

If: by saturation, we can find an $E\in\ns{\mathcal{C}_{X}}$ such
that ${\sim_{X}}\subseteq E$. Then we have that $\ns{X}\times\ns{X}=\mathop{\ns{\bigcup}}_{n\in\ns{\mathbb{N}}}E^{n}$.
By transfer, $X\times X=\bigcup_{n\in\mathbb{N}}F^{n}$ holds for
some standard $F\in\mathcal{C}_{X}$. Therefore, $X$ is coarsely
Archimedean.
\end{proof}
\begin{example}
\label{exa:Uncountable-basis}Let $X$ be an uncountable set equipped
with the coarse structure $\mathcal{C}_{X}:=\set{E\subseteq X^{2}|E\setminus\varDelta_{X}\text{ is countable}}$.
This space is not coarsely Archimedean: for any $E\in\mathcal{C}_{X}$,
since $\bigcup_{n\in\mathbb{N}}E^{n}\setminus\varDelta_{X}$ is countable,
it follows that $\bigcup_{n\in\mathbb{N}}E^{n}\neq X\times X$.
\end{example}

\section{\label{sec:Invariant-iota}Invariant $\iota$}

\subsection{Construction}

Let $\left(X,\xi\right)$ be a pointed coarse space. We define the
invariant $\iota\left(X,\xi\right)$ as follows. Denote the set $\ns{X}\setminus\FIN\left(X,\xi\right)$
by $\INF\left(X,\xi\right)$. For $x,y\in\INF\left(X,\xi\right)$,
we write $x\equiv_{X,\xi}^{\iota}y$ if they lie in the same macrochain-connected
component of $\INF\left(X,\xi\right)$. It is clear that $\equiv_{X,\xi}^{\iota}$
is an equivalence relation on $\INF\left(X,\xi\right)$. For each
$x\in\INF\left(X,\xi\right)$, let $\left[x\right]_{X,\xi}^{\iota}$
denote the $\equiv_{X,\xi}^{\iota}$-equivalence class of $x$. Finally,
define
\[
\iota\left(X,\xi\right):=\Set{\left[x\right]_{X,\xi}^{\iota}|x\in\INF\left(X,\xi\right)}.
\]
In short, $\iota\left(X,\xi\right)$ is the set of all macrochain-connected
components of $\INF\left(X,\xi\right)$.
\begin{prop}[Changing the base point]
Let $X$ be a coarse space. Suppose that two points $\xi,\eta\in X$
lie on the same coarsely connected component. Then $\iota\left(X,\xi\right)=\iota\left(X,\eta\right)$.
\end{prop}

\begin{proof}
By \prettyref{prop:coarse-connectedness}, we have that $\xi\sim_{X}\eta$.
For every $x\in\ns{X}$, if $x\sim_{X}\xi$, then $x\sim_{X}\eta$,
and vice versa. Hence $\INF\left(X,\xi\right)=\INF\left(X,\eta\right)$,
and thus $\iota\left(X,\xi\right)=\iota\left(X,\eta\right)$.
\end{proof}
Because of this, if $X$ is coarsely connected, one can simply write
$\iota\left(X\right)$ for $\iota\left(X,\xi\right)$ dropping the
base point $\xi$ with no ambiguity. The same applies to $\INF\left(X,\xi\right)$
and $\equiv_{X,\xi}^{\iota}$. In particular, one can use this notation
for all metrisable spaces.
\begin{example}
Consider the real line $\mathbb{R}$ (\prettyref{fig:line}).
\begin{figure}
\begin{tikzpicture}
	\draw (3, 0) -- (-3, 0);
	\node[above] at (-3, 0) {$-\infty$};
	\node[above] at (3, 0) {$+\infty$};
	\fill[pattern=north west lines] (-2, -0.2) rectangle (2, 0.2);
\end{tikzpicture}\caption{\label{fig:line}The real line. The shaded region represents the finite
part of the real line. The negative infinities and the positive infinities
are separated from each other by the finite part.}
\end{figure}
 The set $\INF\left(\mathbb{R}\right)$ consists of positive and negative
infinite hyperreals. Two infinite hyperreals are $\equiv_{\mathbb{R}}^{\iota}$-related
if and only if they have the same sign. Hence we have that $\iota\left(\mathbb{R}\right)=\set{\pm\infty}$,
where $\pm\infty$ denote the equivalence classes of positive and
negative infinite hyperreals, respectively. See also \cite[Corollary 15]{MSM10}.
\end{example}

\begin{example}
Consider the plane $\mathbb{R}^{2}$ (\prettyref{fig:plane}).
\begin{figure}
\begin{tikzpicture}
	\fill[pattern=north west lines] (0, 0) circle[radius=2];
	\fill (-2.5, 0) circle[radius=1pt] node[above] {$u$};
	\fill (3, 0) circle[radius=1pt] node[above] {$v$};
	\draw (-2.5, 0) arc (180:360:2.5);
	\draw[->] (2.5, 0) -- (3, 0);
\end{tikzpicture}\caption{\label{fig:plane}The Euclidean plane. The set of all infinite points
of the plane looks like the plane with a single hole.}
\end{figure}
 It is easy to see that any two points $u$ and $v$ of $\INF\left(\mathbb{R}^{2}\right)$
are $\equiv_{\mathbb{R}^{2}}^{\iota}$-related. Hence $\iota\left(\mathbb{R}^{2}\right)=\set{\infty}$.
The same equation holds for all dimensions greater than $2$. This
is contrasted to the case of $\mathbb{R}$.
\end{example}

\begin{example}[{cf. Miller \emph{et al.} \cite[Example 3]{MSM10}}]
Let $X:=\set{\left(\pm1,y\right)|1\leq y}\cup\set{\left(x,1\right)|-1\leq x\leq1}$
(\prettyref{fig:vase}).
\begin{figure}
\begin{tikzpicture}
	\draw[dashed, ->] (-4, 0) -- (+4, 0);
	\node[above] at (0, 4) {$\infty_{\upuparrows}$};
	\draw[dashed, ->] (0, -1) -- (0, 4);
	\draw (-1, 1) -- (1, 1);
	\draw (-1, 1) -- (-1, 4);
	\draw (1, 1) -- (1, 4);
\end{tikzpicture}\caption{\label{fig:vase}The straight vase.}
\end{figure}
 Any two points of $\INF\left(X\right)$ are $\equiv_{X}^{\iota}$-related.
Hence $\iota\left(X\right)=\set{\infty_{\upuparrows}}$. Note that
the ``collapsing'' map $X\ni\left(x,y\right)\mapsto y\in\mathbb{R}_{\geq1}$
gives a coarse equivalence. See also \cite[Corollary 17]{MSM10}.
\end{example}

\begin{example}
Let $X:=\set{\left(\pm y,y\right)|1\leq y}\cup\set{\left(x,1\right)|-1\leq x\leq1}$
(\prettyref{fig:flared-vase}).
\begin{figure}
\begin{tikzpicture}
	\draw[dashed, ->] (-4, 0) -- (+4, 0);
	\draw[dashed, ->] (0, -1) -- (0, 4);
	\draw (-1, 1) -- (1, 1);
	\node[above] at (-4, 4) {$\infty_{\nwarrow}$};
	\node[above] at (4, 4) {$\infty_{\nearrow}$};
	\draw (-1, 1) -- (-4, 4);
	\draw (1, 1) -- (4, 4);
\end{tikzpicture}\caption{\label{fig:flared-vase}The flared vase.}
\end{figure}
 Any two infinite points on the left and the right ``verges'' are
not $\equiv_{X}^{\iota}$-related. Hence $\iota\left(X\right)=\set{\infty_{\nwarrow},\infty_{\nearrow}}$.
Note that the projection $X\ni\left(x,y\right)\mapsto x\in\mathbb{R}$
gives an isomorphism of coarse spaces.
\end{example}

\begin{example}[{cf. DeLyser \emph{et al.} \cite[Example 4.5]{DLT13}}]
Let $X=\set{\left(0,y\right)|y\in\mathbb{R}_{\geq0}}\cup\set{\left(x,2^{n}\right)|x\in\mathbb{R},n\in\mathbb{N}}$
(\prettyref{fig:antenna}).
\begin{figure}
\begin{tikzpicture}
	\draw (0, 0) -- (0, 7);
	\draw (-4, 0.5) -- (4, 0.5);
	\draw (-4, 1) -- (4, 1);
	\draw (-4, 2) -- (4, 2);
	\draw (-4, 4) -- (4, 4);
	\node[above] at (0, 7) {$\infty_{\uparrow}$};
	\node[left] at (-4, 3) {$\infty_{\leftleftarrows}$};
	\node[right] at ( 4, 3) {$\infty_{\rightrightarrows}$};
\end{tikzpicture}\caption{\label{fig:antenna}The antenna.}
\end{figure}
 Let us show that each infinite point $\left(x,y\right)$ of $\ns{X}$
is $\equiv_{X}^{\iota}$-related to one of $\left(-\omega,1\right)$,
$\left(0,\omega\right)$ and $\left(\omega,1\right)$ for a fixed
$\omega\in\ns{\mathbb{N}}\setminus\mathbb{N}$. There are three cases.
\begin{casenv}
\item $y$ is finite and $x$ is negative infinite. Then $y=2^{n}$ for
some (standard) $n\in\mathbb{N}$. Firstly, the internal hyperfinite
sequence
\[
\underset{n+1}{\underbrace{\left(x,y\right),\left(x,2^{-1}y\right),\left(x,2^{-2}y\right),\ldots,\left(x,1\right)}}
\]
connects $\left(x,y\right)$ and $\left(x,1\right)$. Any two adjacent
points are of distance less than or equal to $2^{n-1}$. Secondly,
the internal hyperfinite sequence
\[
\underset{\left\lceil \left|x+\omega\right|\right\rceil +1}{\underbrace{\left(x,1\right),\left(x-\frac{x+\omega}{\left\lceil \left|x+\omega\right|\right\rceil },1\right),\left(x-2\cdot\frac{x+\omega}{\left\lceil \left|x+\omega\right|\right\rceil },1\right),\ldots,\left(-\omega,1\right)}}
\]
connects $\left(x,1\right)$ and $\left(-\omega,1\right)$, where
$\left\lceil \cdot\right\rceil $ is the ceiling function. Any two
adjacent points are of distance less than or equal to $\left|x+\omega\right|/\left\lceil \left|x+\omega\right|\right\rceil \le1$.
Hence $\left(x,y\right)\equiv_{X}^{\iota}\left(x,1\right)\equiv_{X}^{\iota}\left(-\omega,1\right)$.
\item $y$ is finite and $x$ is positive infinite. Similarly to the first
case, we have that $\left(x,y\right)\equiv_{X}^{\iota}\left(\omega,1\right)$.
\item $y$ is infinite. Firstly, the internal hyperfinite sequence
\[
\underset{\left\lceil \left|x\right|\right\rceil +1}{\underbrace{\left(x,y\right),\left(x-\frac{x}{\left\lceil \left|x\right|\right\rceil },y\right),\left(x-2\cdot\frac{x}{\left\lceil \left|x\right|\right\rceil },y\right),\ldots,\left(0,y\right)}}
\]
connects $\left(x,y\right)$ and $\left(0,y\right)$. Any two adjacent
points are of distance less than or equal to $\left|x\right|/\left\lceil \left|x\right|\right\rceil \leq1$.
Secondly, the internal hyperfinite sequence
\[
\underset{\left\lceil \left|\omega-y\right|\right\rceil +1}{\underbrace{\left(0,y\right),\left(0,y+\frac{\omega-y}{\left\lceil \left|\omega-y\right|\right\rceil }\right),\left(0,y+2\cdot\frac{\omega-y}{\left\lceil \left|\omega-y\right|\right\rceil }\right),\ldots,\left(0,\omega\right)}}
\]
connects $\left(0,y\right)$ and $\left(0,\omega\right)$. Any two
adjacent points are of distance less than or equal to $\left|\omega-y\right|/\left\lceil \left|\omega-y\right|\right\rceil \le1$.
Hence $\left(x,y\right)\equiv_{X}^{\iota}\left(0,y\right)\equiv_{X}^{\iota}\left(0,\omega\right)$.
\end{casenv}
On the other hand, any two of $\left(-\omega,1\right)$, $\left(0,\omega\right)$
and $\left(\omega,1\right)$ are not $\equiv_{X}^{\iota}$-related.
\begin{enumerate}
\item To prove $\left(-\omega,1\right)\not\equiv_{X}^{\iota}\left(0,\omega\right)$,
suppose on the contrary that $\left(-\omega,1\right)\equiv_{X}^{\iota}\left(\omega,1\right)$,
i.e., there is an internal hyperfinite sequence $\set{\left(x_{i},y_{i}\right)}_{i\leq n}$
in $\INF\left(X\right)$ such that $\left(x_{0},y_{0}\right)=\left(-\omega,1\right)$,
$\left(x_{n},y_{n}\right)=\left(0,\omega\right)$ and $\left(x_{i},y_{i}\right)\sim_{X}\left(x_{i+1},y_{i+1}\right)$
for all $i<n$. Let $i_{0}>0$ be the smallest hyperinteger such that
$x_{i_{0}}=0$. Since $\left(x_{i_{0}},y_{i_{0}}\right)=\left(0,y_{i_{0}}\right)$
is infinite, $y_{i_{0}}$ must be infinite. Since $\left(x_{i_{0}-1},y_{i_{0}-1}\right)\sim_{X}\left(x_{i_{0}},y_{i_{0}}\right)$,
we have that $y_{i_{0}-1}\sim_{\mathbb{R}}y_{i_{0}}$, so $y_{i_{0}-1}$
is infinite too. For each $i<i_{0}$, since $x_{i}\neq0$ by the choice
of $i_{0}$, it follows that $y_{i}=2^{k_{i}}$ for some $k_{i}\in\ns{\mathbb{N}}$.
Let us show that $y_{i_{0}-j}=y_{i_{0}-1}$ for all $1\leq j\leq i_{0}$
by (internal) induction on $j$. The case $j=1$ is trivial. Suppose
$j>1$. By the induction hypothesis, $2^{k_{i_{0}-j}}=y_{i_{0}-j}\sim_{\mathbb{R}}y_{i_{0}-j+1}=y_{i_{0}-1}=2^{k_{i_{0}-1}}$,
where $k_{i_{0}-j}$ and $k_{i_{0}-1}$ are both infinite. Hence $k_{i_{0}-j}=k_{i_{0}-1}$
and $y_{i_{0}-j}=y_{i_{0}-1}$. (Otherwise, let $k=\min\set{k_{i_{0}-j},k_{i_{0}-1}}$,
then $\left|y_{i_{0}-j}-y_{i_{0}-1}\right|\geq2^{k}=\text{infinite}$,
which contradicts with $y_{i_{0}-j}\sim_{\mathbb{R}}y_{i_{0}-1}$.)
In particular, we have $0=y_{0}=y_{i_{0}-1}\neq0$, a contradiction.
\item With a similar argument, we can prove that $\left(\omega,1\right)\not\equiv_{X}^{\iota}\left(0,\omega\right)$.
\item To prove $\left(-\omega,1\right)\not\equiv_{X}^{\iota}\left(\omega,1\right)$,
suppose on the contrary that $\left(-\omega,1\right)\equiv_{X}^{\iota}\left(\omega,1\right)$,
i.e., there is an internal hyperfinite sequence $\set{\left(x_{i},y_{i}\right)}_{i\leq n}$
in $\INF\left(X\right)$ such that $\left(x_{0},y_{0}\right)=\left(-\omega,1\right)$,
$\left(x_{n},y_{n}\right)=\left(\omega,1\right)$ and $\left(x_{i},y_{i}\right)\sim_{X}\left(x_{i+1},y_{i+1}\right)$
for all $i<n$. Then $\set{x_{i}}_{i\leq n}$ is an internal hyperfinite
sequence in $\ns{\mathbb{R}}$ such that $x_{0}=-\omega$, $x_{n}=\omega$
and $x_{i}\sim_{\mathbb{R}}x_{i+1}$. However, since $\left|\iota\left(\mathbb{R}\right)\right|=1$,
$x_{i_{0}}$ must be finite for some $i_{0}\leq n$. On the other
hand, since $\left(x_{i_{0}},y_{i_{0}}\right)$ is infinite, $y_{i_{0}}$
must be infinite. Hence $\left(-\omega,1\right)\equiv_{X}^{\iota}\left(x_{i_{0}},y_{i_{0}}\right)\equiv_{X}^{\iota}\left(0,\omega\right)$,
a contradiction.
\end{enumerate}
Consequently, $\iota\left(X\right)=\set{\infty_{\leftleftarrows},\infty_{\uparrow},\infty_{\rightrightarrows}}$,
where $\infty_{\leftleftarrows}$, $\infty_{\uparrow}$ and $\infty_{\rightrightarrows}$
denote the equivalence classes of $\left(-\omega,1\right)$, $\left(0,\omega\right)$
and $\left(\omega,1\right)$, respectively.
\end{example}

\subsection{Coarse invariance}
\begin{lem}
Let $f\colon\left(X,\xi\right)\to\left(Y,\eta\right)$ be a proper
map. Then $\ns{f}\colon\ns{\left(X,\xi\right)}\to\ns{\left(Y,\eta\right)}$
maps $\INF\left(X,\xi\right)$ into $\INF\left(Y,\eta\right)$.
\end{lem}

\begin{proof}
By \prettyref{prop:proper}, $\ns{f}^{-1}\left(\ns{Y}\setminus\INF\left(Y,\eta\right)\right)\subseteq\ns{X}\setminus\INF\left(X,\xi\right)$
holds. It follows that $\ns{f}\left(\INF\left(X,\xi\right)\right)\subseteq\INF\left(Y,\eta\right)$.
\end{proof}
\begin{lem}
Let $f\colon\left(X,\xi\right)\to\left(Y,\eta\right)$ be a coarse
(i.e. proper bornologous) map. Then $\ns{f}\colon\ns{\left(X,\xi\right)}\to\ns{\left(Y,\eta\right)}$
sends $\equiv_{X,\xi}^{\iota}$-related pairs to $\equiv_{Y,\eta}^{\iota}$-related
pairs.
\end{lem}

\begin{proof}
Let $x,y\in\INF\left(X,\xi\right)$ be $\equiv_{X,\xi}^{\iota}$-related.
Choose an internal hyperfinite sequence $\set{s_{i}}_{i\leq n}$ in
$\INF\left(X,\xi\right)$ such that $s_{0}=x$, $s_{n}=y$ and $s_{i}\sim_{X}s_{i+1}$
for all $i<n$. Then the internal hyperfinite sequence $\set{\ns{f}\left(s_{i}\right)}_{i\leq n}$
satisfies that $\ns{f}\left(s_{0}\right)=\ns{f}\left(x\right)$, $\ns{f}\left(s_{n}\right)=\ns{f}\left(y\right)$,
and $\ns{f}\left(s_{i}\right)\sim_{Y}\ns{f}\left(s_{i+1}\right)$
for all $i<n$ by \prettyref{prop:bornologous}. Hence $\ns{f}\left(x\right)\equiv_{Y,\eta}^{\iota}\ns{f}\left(y\right)$.
\end{proof}
\begin{thm}[Functoriality]
Every coarse map $f\colon\left(X,\xi\right)\to\left(Y,\eta\right)$
functorially induces a map $\iota f\colon\iota\left(X,\xi\right)\to\iota\left(Y,\eta\right)$.
\end{thm}

\begin{proof}
Define $\iota f\colon\iota\left(X,\xi\right)\to\iota\left(Y,\eta\right)$
by letting $\iota f\left[x\right]_{X,\xi}^{\iota}:=\left[\ns{f}\left(x\right)\right]_{Y,\eta}^{\iota}$.
By the previous lemmas, $\iota f$ is well-defined. Thus $\iota\left({-}\right)$
can be extended to a functor from (an arbitrary small full subcategory
of) the category of coarse spaces to the category of sets.
\end{proof}
\begin{cor}
$\iota$ is invariant under isomorphisms of coarse spaces.
\end{cor}

\begin{thm}[Coarse invariance]
If two coarse maps $f,g\colon\left(X,\xi\right)\to\left(Y,\eta\right)$
are bornotopic, then $\iota f=\iota g$.
\end{thm}

\begin{proof}
Let $\left[x\right]_{X,\xi}^{\iota}\in\INF\left(X,\xi\right)$. By
\prettyref{prop:bornotopy}, $\ns{f}\left(x\right)\sim_{Y}\ns{g}\left(x\right)$,
so $\ns{f}\left(x\right)\equiv_{Y,\eta}\ns{g}\left(x\right)$. Hence
$\iota f\left[x\right]_{X,\xi}^{\iota}=\iota g\left[x\right]_{X,\xi}^{\iota}$.
\end{proof}
\begin{cor}
$\iota$ is invariant under coarse equivalences.
\end{cor}

\subsection{Wedge sums}

Let $\left(X,\xi\right)$ and $\left(Y,\eta\right)$ be pointed sets.
The wedge sum $\left(X\vee Y,p\right)$ is the quotient set $\left(X\sqcup Y\right)/E$,
where $E$ is the equivalence closure of $\set{\left(\xi,\eta\right)}$,
with the base point $p=\left[\xi\right]=\left[\eta\right]$. Suppose
$X$ and $Y$ are coarse spaces. The wedge sum $X\vee Y$ is then
equipped with the coarse structure whose finite closeness relation
is described as
\[
u\sim_{X\vee Y}v\iff\begin{cases}
u\sim_{X}v, & u,v\in\ns{X},\\
u\sim_{Y}v, & u,v\in\ns{Y},\\
u\sim_{X}\xi\text{ and }\eta\sim_{Y}v, & u\in\ns{X},v\in\ns{Y},\\
u\sim_{Y}\eta\text{ and }\xi\sim_{X}v, & u\in\ns{Y},v\in\ns{X}.
\end{cases}
\]
For instance, if $X$ and $Y$ are metrisable, then the coarse structure
of $X\vee Y$ coincides with the coarse structure induced by the following
metric:
\[
d_{X\vee Y}\left(u,v\right):=\begin{cases}
d_{X}\left(u,v\right), & u,v\in X,\\
d_{Y}\left(u,v\right), & u,v\in Y,\\
d_{X}\left(u,\xi\right)+d_{Y}\left(\eta,v\right), & u\in X,v\in Y,\\
d_{Y}\left(u,\eta\right)+d_{X}\left(\xi,v\right), & u\in Y,v\in X.
\end{cases}
\]
We consider $\left(X\vee Y,p\right)$ as a pointed coarse space equipped
with the coarse structure above.
\begin{lem}
Let $\left(X,\xi\right)$ and $\left(Y,\eta\right)$ be pointed coarse
spaces. Then $\INF\left(X\vee Y,p\right)=\INF\left(X,\xi\right)\sqcup\INF\left(Y,\eta\right)$.
\end{lem}

\begin{proof}
For every $\left[u\right]\in\ns{\left(X\vee Y\right)}$, we have that
\begin{align*}
u\in\FIN\left(X\vee Y,p\right) & \iff\left[u\right]\sim_{X\vee Y}p\\
 & \iff\left(u\in\ns{X}\text{ and }u\sim_{X}\xi\right)\text{ or }\left(u\in\ns{Y}\text{ and }u\sim_{Y}\eta\right)\\
 & \iff u\in\FIN\left(X,\xi\right)\text{ or }u\in\FIN\left(Y,\eta\right),
\end{align*}
thereby $\INF\left(X\vee Y,p\right)=\INF\left(X,\xi\right)\sqcup\INF\left(Y,\eta\right)$.
\end{proof}
\begin{prop}
Let $\left(X,\xi\right)$ and $\left(Y,\eta\right)$ be pointed coarse
spaces. Then $\iota\left(X\vee Y,p\right)=\iota\left(X,\xi\right)\sqcup\iota\left(Y,\eta\right)$.
\end{prop}

\begin{proof}
Immediate from the above lemma.
\end{proof}
\begin{example}
The pointed real line $\left(\mathbb{R},0\right)$ is isomorphic to
the wedge sum $\left(\mathbb{R}_{\leq0},0\right)\vee\left(\mathbb{R}_{\geq0},0\right)$.
So $\iota\left(\mathbb{R},0\right)=\iota\left(\mathbb{R}_{\leq0},0\right)\sqcup\iota\left(\mathbb{R}_{\geq0},0\right)=\set{-\infty}\sqcup\set{+\infty}$.
\end{example}

\subsection{Non-scattering property}

Let $\left(X,\xi\right)$ be a pointed coarse space. Denote the set
of bounded sets of $X$ containing $\xi$ by $\mathcal{B}_{X}\left(\xi\right)$
and consider it as a partially ordered set with respect to the inclusion
$\subseteq$. A subset $A$ of $X$ is said to be \emph{$E$-connected}
for $E\in\mathcal{C}_{X}$ if $X\times X=\bigcup_{n\in\mathbb{N}}E^{n}$.
We say that $\left(X,\xi\right)$ is \emph{non-scattering at infinity}
if there exists an $E\in\mathcal{C}_{X}$ such that the set $\set{B\in\mathcal{B}_{X}\left(\xi\right)|X\setminus B\text{ is \ensuremath{E}-connected}}$
is cofinal in $\mathcal{B}_{X}\left(\xi\right)$, i.e. if for any
$A\in\mathcal{B}_{X}\left(\xi\right)$ there is a $B\supseteq A$
such that $X\setminus B$ is $E$-connected.
\begin{lem}
If a pointed coarse space $\left(X,\xi\right)$ is non-scattering
at infinity, then $\INF\left(X,\xi\right)$ is macrochain-connected.
\end{lem}

\begin{proof}
Let $E$ be a witness of the non-scattering property of $\left(X,\xi\right)$.
Let $x,y\in\INF\left(X,\xi\right)$. For each $A\in\mathcal{B}_{X}\left(\xi\right)$
consider the set $\mathcal{F}_{A}$ of all $B\in\ns{\mathcal{B}_{X}}\left(\xi\right)$
such that $\ns{A}\subseteq B$, $x,y\notin B$, and $\ns{X}\setminus B$
is internally $\ns{E}$-connected. Then the family $\set{\mathcal{F}_{A}|A\in\mathcal{B}_{X}\left(\xi\right)}$
has the finite intersection property, i.e. any finite intersection
from that family is non-empty. Hence, by saturation, we can find a
$B\in\ns{\mathcal{B}_{X}}\left(\xi\right)$ such that $\FIN\left(X,\xi\right)\subseteq B$,
$x,y\notin B$, and $\ns{X}\setminus B$ is internally $\ns{E}$-connected.
(It can be taken as an element of the intersection $\bigcap_{A\in\mathcal{B}_{X}}\mathcal{F}_{A}$.)
There exists an internal hyperfinite sequence $\set{s_{i}}_{i\leq n}$
in $\ns{X}\setminus B$ such that $s_{0}=x$, $s_{n}=y$ and $\left(s_{i},s_{i+1}\right)\in\ns{E}$
for all $i<n$. Since $\ns{E}\subseteq{\sim_{X}}$ and $\ns{X}\setminus B\subseteq\INF\left(X,\xi\right)$,
the sequence $\set{s_{i}}_{i\leq n}$ is included in $\INF\left(X,\xi\right)$,
and satisfies that $s_{i}\sim_{X}s_{i+1}$ for all $i<n$. Hence $\INF\left(X,\xi\right)$
is macrochain-connected.
\end{proof}
\begin{prop}
If a pointed coarse space $\left(X,\xi\right)$ is non-scattering
at infinity, then $\iota\left(X,\xi\right)$ is an empty set or a
singleton.
\end{prop}

\begin{proof}
Immediate from the above lemma.
\end{proof}

\section{\label{sec:Comparison-with-sigma}Comparison with the invariant $\sigma$}

\subsection{Natural transformation $\omega\colon\sigma\to\iota$}

We first recall the invariant $\sigma\left(X,\xi\right)$ of a pointed
metric space $\left(X,\xi\right)$ defined in \cite{DLW11}. We adopt
a simplified but equivalent definition given in \cite{DLT13}.

A \emph{coarse sequence} on $X$ based at $\xi$ is a coarse map $s\colon\left(\mathbb{N},0\right)\to\left(X,\xi\right)$.
Obviously every coarse sequence tends from $\xi$ to infinity. We
denote by $S\left(X,\xi\right)$ the set of all coarse sequences on
$X$ based at $\xi$. Given two coarse sequences $s,t\in S\left(X,\xi\right)$,
we write $s\sqsubseteq t$ if $s$ is a subsequence of $t$, i.e.,
there is a strictly monotone function $\kappa\colon\mathbb{N}\to\mathbb{N}$
such that $s=t\circ\kappa$. We denote by $\equiv_{X,\xi}^{\sigma}$
the smallest equivalence relation on $S\left(X,\xi\right)$ that contains
$\sqsubseteq$. In other words, $s\equiv_{X,\xi}^{\sigma}t$ if and
only if $s$ and $t$ can be obtained from each other by taking (coarse)
subsequences and supersequences repeatedly. Finally define
\[
\sigma\left(X,\xi\right):=\set{\left[s\right]_{X,\xi}^{\sigma}|s\in S\left(X,\xi\right)},
\]
where $\left[s\right]_{X,\xi}^{\sigma}$ denotes the $\equiv_{X,\xi}^{\sigma}$-equivalence
class of $s$. The definition can be extended to arbitrary pointed
coarse spaces by simply replacing `metric' with `coarse'.

The definitions of $\sigma$ and $\iota$ look similar in the following
sense. The standard invariant $\sigma$ was constructed from (bornologous)
sequences tending to infinity. These sequences could be thought of
as \emph{dynamic} infinities in the sense that such infinities are
captured through the limit process. On the other hand, the nonstandard
invariant $\iota$ was constructed from nonstandardly infinite points.
These points could be thought of as \emph{static} infinities. In spite
of their similarity, $\sigma\left(X\right)$ and $\iota\left(X\right)$
do not coincide for some metric spaces as we shall see later.

In order to compare those two invariants, let us construct a natural
transformation from $\sigma$ to $\iota$. We first prove the following
lemma.
\begin{lem}
Let $\left(X,\xi\right)$ be a coarse space, $s,t\in S\left(X,\xi\right)$
and $i,j\in\prescript{\ast}{}{\mathbb{N}}\setminus\mathbb{N}$. Then
$\ns{s_{i}},\ns{t_{j}}$ belong to $\INF\left(X,\xi\right)$. If $s\equiv_{X,\xi}^{\sigma}t$,
then $\ns{s_{i}}\equiv_{X,\xi}^{\iota}\ns{t_{j}}$.
\end{lem}

\begin{proof}
It is not difficult to prove that a sequence $u$ in $X$ tends (from
$\xi$) to infinity if and only if $\ns{u_{k}}\in\INF\left(X,\xi\right)$
holds for all $k\in\ns{\mathbb{N}}\setminus\mathbb{N}=\INF\left(\mathbb{N},0\right)$.
Hence $\ns{s_{i}},\ns{t_{j}}$ belong to $\INF\left(X,\xi\right)$.

To prove the latter part, we only need to prove the case $s\sqsubseteq t$,
because the general case can be reduced to that case by induction.
Choose a strictly monotone function $\kappa\colon\mathbb{N}\to\mathbb{N}$
such that $s=t\circ\kappa$. By transfer, this holds also at $i$,
i.e. $\ns{s_{i}}=\ns{t_{\ns{\kappa}\left(i\right)}}$. So $\ns{s_{i}}$
and $\ns{t_{j}}$ can be connected by an internal hyperfinite sequence
$\ns{t_{\ns{\kappa}\left(i\right)}},\ldots,\ns{t_{j}}$. Since $t$
is coarse, this sequence lies in $\INF\left(X,\xi\right)$ and any
two adjacent points are finitely close. Hence $\ns{s_{i}}\equiv_{X,\xi}^{\iota}\ns{t_{j}}$
holds.
\end{proof}
Thanks to this lemma, we can well-define a map $\omega_{\left(X,\xi\right)}\colon\sigma\left(X,\xi\right)\to\iota\left(X,\xi\right)$
by letting $\omega_{\left(X,\xi\right)}\left[s\right]_{X,\xi}^{\sigma}:=\left[\ns{s_{\omega}}\right]_{X,\xi}^{\iota}$,
where $\omega\in\prescript{\ast}{}{\mathbb{N}}\setminus\mathbb{N}$.
Clearly the map $\omega_{\left(X,\xi\right)}$ is natural in $\left(X,\xi\right)$,
i.e. it gives a natural transformation from $\sigma$ to $\iota$.

\subsection{Surjectivity for proper geodesic spaces}

We use the notation and terminology of nonstandard small-scale topology.
Recall that a metric space is said to be \emph{proper} if every closed
bounded subset is compact. A metric space is proper if and only if
every finite point is nearstandard \cite[Theorem 5.6 of Chapter 3]{Dav05}.
\begin{thm}
Let $X$ be a proper geodesic metric space and $\xi\in X$. The map
$\omega_{\left(X,\xi\right)}\colon\sigma\left(X,\xi\right)\to\iota\left(X,\xi\right)$
is surjective.
\end{thm}

\begin{proof}
Let $x\in\INF\left(X,\xi\right)$. Since $X$ is geodesic, there is
an internal isometry $\gamma\colon\left[0,\ns{d_{X}}\left(\xi,x\right)\right]\to\ns{X}$
such that $\gamma\left(0\right)=\xi$ and $\gamma\left(\ns{d_{X}}\left(\xi,x\right)\right)=x$
by transfer. For each $i\in\mathbb{N}$, since $\ns{d_{X}}\left(\xi,\gamma\left(i\right)\right)=i$
is finite, $\gamma\left(i\right)\in\FIN\left(X,\xi\right)$, so $\gamma\left(i\right)$
is nearstandard by properness. We can take a sequence $\set{s_{i}}_{i\in\mathbb{N}}$
in $X$ such that $s_{i}\approx_{X}\gamma\left(i\right)$, where $\approx$
denotes the infinitesimal closeness relation. Since $d_{X}\left(s_{i},s_{i+1}\right)\approx_{\mathbb{R}}\ns{d_{X}}\left(\gamma\left(i\right),\gamma\left(i+1\right)\right)=1$,
the sequence $\set{s_{i}}_{i\in\mathbb{N}}$ is bornologous. Moreover,
it is proper (i.e. tends to infinity), because $d_{X}\left(\xi,s_{i}\right)\approx_{\mathbb{R}}\ns{d_{X}}\left(\gamma\left(0\right),\gamma\left(i\right)\right)=i\to\infty$
as $i\to\infty$. Therefore $s\in S\left(X,\xi\right)$. Furthermore,
since $\ns{s_{i}}\approx_{X}\gamma\left(i\right)$ holds for all $i\in\mathbb{N}$,
it holds also for some $i\in\ns{\mathbb{N}}\setminus\mathbb{N}$ by
Robinson's lemma \cite[Theorem 4.3.10]{Rob66}. For such $i\in\ns{\mathbb{N}}\setminus\mathbb{N}$,
the internal hyperfinite sequence $\ns{s_{i}}\approx_{X}\gamma\left(i\right),\gamma\left(i+1\right),\ldots,\gamma\left(\left\lfloor \ns{d_{X}}\left(\xi,x\right)\right\rfloor \right),\gamma\left(\ns{d_{X}}\left(\xi,x\right)\right)=x$
witnesses that $\left[\ns{s_{i}}\right]_{X,\xi}^{\iota}=\left[x\right]_{X,\xi}^{\iota}$.
Therefore $\omega_{\left(X,\xi\right)}\left[s\right]_{X,\xi}^{\sigma}=\left[x\right]_{X,\xi}^{\iota}$.
\end{proof}

\subsection{Non-surjective examples}

The map $\omega_{\left(X,\xi\right)}\colon\sigma\left(X,\xi\right)\to\iota\left(X,\xi\right)$
is bijective for \emph{some} coarse spaces, e.g. the Euclidean spaces,
the ``vase'' spaces and the ``antenna'' space which are presented
in the previous section. It is however not bijective in general. The
following are examples where $\omega_{\left(X,\xi\right)}$ is not
surjective.
\begin{example}
\label{exa:Square-number-space}Consider the subspace $X:=\set{n^{2}|n\in\mathbb{N}}$
of the real line $\mathbb{R}$ with an arbitrary base point $\xi$.
This space has no coarse sequence. So $\sigma\left(X,\xi\right)$
is empty. On the other hand, since $X$ is unbounded, $\iota\left(X,\xi\right)$
is nonempty. More precisely, $\iota\left(X,\xi\right)=\set{\left[n^{2}\right]_{X,\xi}^{\iota}|n\in\ns{\mathbb{N}}\setminus\mathbb{N}}\cong\ns{\mathbb{N}}\setminus\mathbb{N}\cong\ns{\mathbb{N}}$.
Hence $\sigma\left(X,\xi\right)$ and $\iota\left(X,\xi\right)$ are
not equipotent. In this case, the map $\omega_{\left(X,\xi\right)}\colon\sigma\left(X,\xi\right)\to\iota\left(X,\xi\right)$
is trivially injective but not surjective.
\end{example}

\begin{example}
Recall the coarse space $X$ of \prettyref{exa:Uncountable-basis}.
It is clear that every countable subset of $X$ is bounded, and vice
versa. Hence $X$ has no divergent \emph{sequence}, although it has
a divergent \emph{net}. Thus $\sigma\left(X,\xi\right)=\varnothing$
but $\iota\left(X,\xi\right)\neq\varnothing$ for any base point $\xi\in X$.
\end{example}

\begin{example}[{cf. DeLyser \emph{et al.} \cite[pp. 11--12]{DLW11}}]
\label{exa:Open-books}Consider two metric spaces, called the open
book and the discrete open book (\prettyref{fig:open-books}).
\begin{figure}
\subfloat[The open book.]{\begin{tikzpicture}
	\foreach \n in {-4, -2, 0, 2, 4}
		\draw (0, 0) -- (\n, 4);
\end{tikzpicture}}

\subfloat[The discrete open book.]{\begin{tikzpicture}
	\fill (0, 0) circle[radius=0.5pt];
	\foreach \n in {1, ..., 120}
		\fill ({- \n / 30}, {\n / 30}) circle[radius=0.5pt];
	\foreach \n in {1, ..., 60}
		\fill ({- \n / 30}, {\n / 15}) circle[radius=0.5pt];
	\foreach \n in {1, ..., 40}
		\fill (0, {\n / 10}) circle[radius=0.5pt];
	\foreach \n in {1, ..., 30}
		\fill ({\n / 15}, {\n / 7.5}) circle[radius=0.5pt];
	\foreach \n in {1, ..., 24}
		\fill ({\n / 6}, {\n / 6}) circle[radius=0.5pt];
\end{tikzpicture}}\caption{\label{fig:open-books}The first $5$ pages of the books.}
\end{figure}
 The open book $B$ is the wedge sum of countable copies of the ray
$\mathbb{R}_{\geq0}$ at the origin $0$. Let $s^{i}$ be the coarse
sequence $0,1,2,\ldots$ on the $i$-th ``page'' $\left[0,+\infty\right)$
of $B$. Then,
\[
\sigma\left(B,0\right)=\Set{\left[s^{i}\right]_{B,0}^{\sigma}|i=1,2,\ldots}.
\]
On the other hand, the discrete open book $D$ is the wedge sum of
$i\mathbb{N}:=\set{in|n\in\mathbb{N}}\ \left(i=1,2,\ldots\right)$
at the origin $0$, which is a subspace of $B$. Let $t^{i}$ be the
coarse sequence $0,i,2i,\ldots$ on the $i$-th page $i\mathbb{N}$
of $D$. Then,
\[
\sigma\left(D,0\right)=\Set{\left[t^{i}\right]_{D,0}^{\sigma}|i=1,2,\ldots}.
\]
Thus $\sigma\left(B,0\right)\cong\sigma\left(D,0\right)\cong\mathbb{N}$.

Let us compute the nonstandard invariants $\iota$ of the two books
$B$ and $D$. In both books, any two infinite points on different
pages are inequivalent, but any two infinite points on the same finite-numbered
page are equivalent. In the open book $B$, the latter remains true
on infinite-numbered pages. Hence 
\[
\iota\left(B,0\right)=\Set{\left[\ns{s}_{\omega}^{i}\right]_{B,0}^{\iota}|i\in\ns{\mathbb{N}_{>0}}},
\]
where $\omega\in\ns{\mathbb{N}}\setminus\mathbb{N}$. In the discrete
open book $D$, any two different points on infinite-numbered pages
are inequivalent even when they are on the same page. Hence 
\[
\iota\left(D,0\right)=\Set{\left[\ns{t}_{\omega}^{i}\right]_{D,0}^{\iota}|i\in\mathbb{N}_{>0}}\cup\Set{\left[\ns{t}_{j}^{i}\right]_{D,0}^{\iota}|i\in\ns{\mathbb{N}}\setminus\mathbb{N},j\in\ns{\mathbb{N}_{>0}}},
\]
where $\omega\in\ns{\mathbb{N}}\setminus\mathbb{N}$. Both $\iota\left(D,0\right)$
and $\iota\left(B,0\right)$ are equipotent to $\prescript{\ast}{}{\mathbb{N}}$,
i.e., $\iota\left(B,0\right)\cong\iota\left(D,0\right)\cong\ns{\mathbb{N}}$.

Now consider the following diagram:
\[
\xymatrix{\sigma\left(D,0\right)\ar[d]^{\omega_{\left(D,0\right)}}\ar[rr]^{\sigma\left(D\hookrightarrow B\right)} &  & \sigma\left(B,0\right)\ar[d]^{\omega_{\left(B,0\right)}}\\
\iota\left(D,0\right)\ar[rr]^{\iota\left(D\hookrightarrow B\right)} &  & \iota\left(B,0\right)
}
\]
Both of the vertical maps are injective, but neither is surjective.
The upper horizontal map is bijective, because it sends $\left[t^{i}\right]_{D,0}^{\sigma}$
to $\left[s^{i}\right]_{B,0}^{\sigma}$ for all $i\in\mathbb{N}$.
On the other hand, the lower horizontal map is surjective but not
injective, because it sends $\left[\ns{t}_{j}^{i}\right]_{D,0}^{\iota}$
to $\left[\ns{s}_{\omega}^{i}\right]_{B,0}^{\iota}$ for all $i\in\ns{\mathbb{N}}\setminus\mathbb{N}$
and $j\in\ns{\mathbb{N}_{>0}}$.
\end{example}

\section{\label{sec:Some-open-problems}Some open problems}

It is easy to construct, for each $n\in\mathbb{N}$, a pointed coarse
space $\left(X_{n},\xi_{n}\right)$ such that $\iota\left(X_{n},\xi_{n}\right)$
is of cardinality $n$ (e.g. the wedge sum of $n$ copies of the ray
$\mathbb{R}_{\geq0}$ at $0$). We also have a pointed coarse space
$\left(Y,\eta\right)$ such that $\iota\left(Y,\eta\right)$ is equipotent
to the hypernatural numbers $\ns{\mathbb{N}}$ (e.g. \prettyref{exa:Square-number-space}
and \prettyref{exa:Open-books}). The cardinality of $\ns{\mathbb{N}}$
is uncountable, because each function $f\colon\mathbb{N}\to\set{0,1}$
can be coded by $\sum_{i=0}^{\omega}2^{i}\ns{f}\left(i\right)\in\ns{\mathbb{N}}$,
where $\omega\in\prescript{\ast}{}{\mathbb{N}}\setminus\mathbb{N}$.
\begin{rem}
The cardinality of $\ns{\mathbb{N}}$ is greater than that of any
standard set by weak saturation (i.e. the enlargement property). Moreover,
if the nonstandard universe is $\kappa$-saturated, the cardinality
of $\ns{\mathbb{N}}$ is at least $\kappa$. Thus the cardinality
of $\ns{\mathbb{N}}$ depends on the choice of the standard and nonstandard
universes.
\end{rem}

\begin{problem}
Does there exist a pointed coarse space $\left(X,\xi\right)$ such
that $\iota\left(X,\xi\right)$ is countably infinite?
\end{problem}

We have compared two invariants $\sigma\left(X,\xi\right)$ and $\iota\left(X,\xi\right)$
through the natural transformation $\omega_{\left(X,\xi\right)}\colon\sigma\left(X,\xi\right)\to\iota\left(X,\xi\right)$.
As already mentioned, this map is sometimes bijective, and is sometimes
injective but not surjective. The following questions then arise.
\begin{problem}
Does there exist a pointed coarse space $\left(X,\xi\right)$ such
that the map $\omega_{\left(X,\xi\right)}$ is not injective?
\end{problem}

\begin{problem}
For what kind of pointed coarse spaces is the map $\omega_{\left(X,\xi\right)}$
injective, surjective, bijective?
\end{problem}

The invariant $\sigma\left(X,\xi\right)$ only deals with standard
\emph{bornologous} sequences, so it fails to capture infinite points
that cannot be reached in a standard bornologous way. On the other
hand, the invariant $\iota\left(X,\xi\right)$ deals with arbitrary
infinite points. So $\iota\left(X,\xi\right)$ may have richer elements
compared with $\sigma\left(X,\xi\right)$. \prettyref{exa:Square-number-space}
and \prettyref{exa:Open-books} exemplify such circumstances. Thus
those two invariants may focus on quite different properties of spaces.
In fact, the surjectivity of $\omega_{\left(X,\xi\right)}$ can be
recovered by reducing the elements of $\mathrm{INF}\left(X,\xi\right)$.
Consider the following subset of $\INF\left(X,\xi\right)$:
\[
\INF^{b}\left(X,\xi\right):=\set{x\in\INF\left(X,\xi\right)|\exists s\in S\left(X,\xi\right)\exists\omega\in\ns{\mathbb{N}}\left(x=\ns{s}_{\omega}\right)}.
\]
Then define an invariant $\iota^{b}\left(X,\xi\right)$ as the set
of all macrochain-connected components of $\INF^{b}\left(X,\xi\right)$.
One can define a natural map $\omega_{\left(X,\xi\right)}^{b}\colon\sigma\left(X,\xi\right)\to\iota^{b}\left(X,\xi\right)$
in the same way as $\omega_{\left(X,\xi\right)}$. Obviously $\omega_{\left(X,\xi\right)}^{b}$
is surjective. If $\omega_{\left(X,\xi\right)}$ is injective, then
so is $\omega_{\left(X,\xi\right)}^{b}$, because each macrochain-connected
component of $\INF^{b}\left(X,\xi\right)$ is contained in some macrochain-connected
component of $\INF\left(X,\xi\right)$. However, the injectivity of
$\omega_{\left(X,\xi\right)}^{b}$ is still non-trivial.
\begin{problem}
For what kind of pointed coarse spaces is the map $\omega_{\left(X,\xi\right)}^{b}$
injective (and hence bijective)?
\end{problem}

Because the invariant $\sigma\left(X,\xi\right)$ is made out of coarse
\emph{sequences}, it may be well-behaved only for coarse spaces with
countable bases. The existence of countable bases is equivalent to
metrisability provided that the metric function is allowed to take
the value $+\infty$ (see \cite[Theorem 2.55]{Roe03}).
\begin{problem}
Find a more appropriate definition of the invariant $\sigma\left(X,\xi\right)$
for non-metrisable spaces.
\end{problem}

The author in \cite{Ima16} constructed a nonstandard homology theory
for uniform spaces based on hyperfinite formal sums of \emph{infinitesimally
small} simplices.\footnote{To be fair, this idea is due to McCord \cite{McC72}, who developed
a nonstandard homology theory of topological spaces.} Similarly, one can construct a nonstandard homology theory for coarse
spaces based on hyperfinite formal sums of \emph{finitely large} simplices.
Then $\iota\left(X,\xi\right)$ can be regarded as a ``basis'' of
the $0$-th homology group of $\INF\left(X,\xi\right)$. Note that
the set $\iota\left(X,\xi\right)$ is NOT a basis in the usual sense.
Strictly speaking, the $0$-th homology group of $\INF\left(X,\xi\right)$
consists of hyperfinite formal sums $\sum_{i}g_{i}\sigma_{i}$, where
$g_{i}$ are elements of the coefficient group and $\sigma_{i}=\left[x_{i}\right]_{X,\xi}^{\iota}$
are elements of $\iota\left(X,\xi\right)$ such that the sequence
$\set{\left(g_{i},x_{i}\right)}_{i}$ is internal.
\begin{problem}
Can we interpret the invariant $\sigma\left(X,\xi\right)$ homology-theoretically?
\end{problem}

\bibliographystyle{amsplain}
\bibliography{\string"A nonstandard invariant of coarse spaces\string"}

\end{document}